\documentclass{amsart}
 \usepackage{amsthm,amsfonts,amsmath,amssymb,latexsym,epsfig,upref,eucal,ae}
\usepackage[all]{xy}

\usepackage{graphicx}

\usepackage{subfigure}
\usepackage{epic}
\usepackage{eepic}
\usepackage{setspace}
\usepackage{booktabs}
\usepackage{longtable}
\usepackage{pstricks}
\usepackage{url}

\newtheorem{theorem}[subsubsection]{Theorem}
\newtheorem{lemma}[subsubsection]{Lemma}
\newtheorem{proposition}[subsubsection]{Proposition}
\newtheorem{corollary}[subsubsection]{Corollary}
\newenvironment{example}{\medskip \refstepcounter{subsubsection}
\noindent  {\bf Example \thetheorem}.\rm}{\,}
\newenvironment{remark}{\medskip \refstepcounter{subsubsection}
\noindent  {\bf Remark \thetheorem}.\rm}{\,}

\newtheorem{quest}[subsubsection]{Question}

\newtheorem{theointro}{Theorem}

\newtheorem{theoin}{Proposition}

\newtheorem{corintro}[theoin]{Corollary}
\newtheorem{questintro}[theoin]{Question}
\newtheorem{conjintro}[theoin]{Conjecture}

\newtheorem{remarkintro}{Remark}

\def\Spec{\mathrm{Spec}}

\def\om{\omega}
\def\Vol{\mathrm{Vol}}
\def\Aut{\mathrm{Aut}}

\def\RR{\mathbb{R}}
\def\ZZ{\mathbb{Z}}
\def\NN{\mathbb{N}}

\def\FF{\mathbb{F}}

\def\rmBl{\mathrm{Bl}}

\def\CC{\mathbb{C}}
\def\PP{\mathbb{P}}

\def\ep{\varepsilon}

\def\>{\rangle}

\def\<{\langle}

\begin{document}

\title[A note on blow-ups of toric surfaces and CSC K\"ahler metrics]
{A note on blow-ups of toric surfaces and CSC K\"ahler metrics}
\author[C. Tipler]{Carl Tipler}
\address{D\'epartment de Math\'ematiques, Laboratoire Jean Leray,
2, Rue de la Houssini\`ere - BP 92208, F-44322 Nantes, FRANCE}
\email{carl.tipler@univ-nantes.fr}


\begin{abstract}
Let $X$ be a compact toric surface. Then there exists 
a sequence of torus equivariant blow-ups of $X$ such that the blown-up toric surface admits a cscK metric.
\end{abstract}

\maketitle

\section{Introduction}
\label{secintro}

The problem of finding constant scalar curvature K\"ahler metrics (cscK for short) on
compact K\"ahler manifolds is of main interest in K\"ahler geometry since the works of 
Yau \cite{yau}, Tian \cite{tian} and Donaldson \cite{don}. The Yau-Tian-Donaldson
conjecture asserts that there exists a cscK metric on a polarized
K\"ahler manifold $(X,L)$ in the K\"ahler class $c_1(L)$ if and only if the pair $(X,L)$
is stable in a suitable GIT sense. At that time, we know that the existence
of a cscK metric implies the $K$-polystability by a result of Stoppa \cite{st}.
Stoppa's argument relies on an important result of Arezzo and Pacard on 
blow-ups of cscK manifolds \cite{ap1}, \cite{ap2}. Arezzo and Pacard managed
to show that if $(X,\om)$ is cscK, then the blow-up of $X$ at finitely many well-chosen 
points admits a cscK metric. 
In order to prove their result, they use a gluing method. They build the blown-up manifold by
gluing a local model, the blow-up of $\CC^n$ at $0$, endowed with the Burns-Simanca metric \cite{sim},
which is ALE and scalar flat.

Another result on blow-ups and special metrics is the one of Taubes \cite{taubes} that asserts that there
are anti-self-dual metrics on the blow-up at sufficiently many points of any smooth oriented $4$-manifold.
In the K\"ahler case, a metric is anti-self-dual if and only if it is scalar-flat. 

These two results suggest the following:

\begin{conjintro}
\label{conj}
 Let $X$ be a compact K\"ahler manifold. Then there exists a finite sequence
of blow-ups of $X$ such that the blown-up manifold obtained admits a cscK metric.
\end{conjintro}

In the case of ruled surfaces, this conjecture has been proved by Le\,Brun and Singer \cite{lesi}
when the base curve is of genus greater than or equal to $2$ and then in general by Kim, Le\,Brun and Pontecorvo \cite{klp}.
However, in their argument, even if the starting surface is toric, the resulting blown-up surface
does not admit the torus action anymore.

On the other hand, Rollin and Singer found another method to built cscK metrics on the blown-up surfaces,
based on parabolic structures \cite{rs}, from which we can deduce a lot of examples of toric surfaces with cscK metrics.

In this paper, we show that if the initial surface is toric, then we can prove the conjecture within the toric world. 

\begin{theointro}
\label{theor}
 Let $X$ be a compact complex smooth toric surface. Then there exists a finite sequence of blow-ups
of $X$ at torus fixed points such that the blown-up toric surface admits a cscK metric.
\end{theointro}

\begin{remarkintro}
\rm
 By a remark of Wang and Zhou \cite{wz}, using the theorem of Arezzo, Pacard and Singer 
on blow-ups of extremal metrics \cite{aps},
we know that each toric surface admits an extremal metric. However, from Matsushima and Lichnerowicz obstruction, we know from Calabi \cite{ca} that the Hirzebruch surfaces admit no extremal metric of constant scalar curvature.
Thus we need to blow-up toric surfaces if we want them to admit cscK metrics.
\end{remarkintro}

\begin{remarkintro}
\rm
About the existence of K\"ahler-Einstein metrics on toric manifolds, a recent result of Legendre \cite{Le}
shows that each toric orbifold admits a unique singular K\"ahler-Einstein metric.
\end{remarkintro}

The proof of this result relies on simple considerations on the
classification of toric surfaces and the result of Arezzo and Pacard on blow-ups of cscK metrics.
We hope that it will outline a stabilization process of blowing up a manifold at suitable points and its link
with the existence of cscK metrics.
Basically, the idea is to blow up the toric surface until it has a good shape.
We first build an infinite sequence of toric surfaces 
$$
 ... \rightarrow X_{j+1} \rightarrow X_j \rightarrow ... \rightarrow X_1 \rightarrow X_0
$$
with $X_0=\PP^1\times\PP^1$
and $X_{j+1}$ obtained from $X_j$ by blowing up all the torus fixed points.
By the theorem of Arezzo and Pacard, each of these $X_j$ admits a cscK metric.

Then, using the classification of toric surfaces, we prove that for each smooth toric surface $X$,
there exists a sequence of blow-ups at torus fixed points such that the blown-up surface is
one of the $X_j's$. From the fan description point of view, we add rays to the fan describing $X$ until
it becomes very symmetric, isomorphic to the fan of an $X_j$. In fact it turns out that 
the $X_j$ endowed with a suitable polarization are $K$-polystable and asymptotically Chow polystable.

We will deduce from our result the following corollary:

\begin{corintro}
\label{cori}
  Let $X$ be a smooth compact toric surface. Then there is a sequence of toric blow-ups of $X$ $($see Section~\ref{blup}$)$
$$\mathrm{Bl}_{(p_m,\ldots,p_1)}(X) \rightarrow X$$ and a polarization
$L_X \rightarrow \rmBl_{(p_m,\ldots,p_1)}(X)$ such that $(\rmBl_{(p_m,\ldots,p_1)}(X),L_X)$ is $K$-polystable and asymptotically Chow-polystable.
\end{corintro}

\begin{remarkintro}
\rm
 Donaldson has proved that a polarized toric surface $(X,L)$ admits
a cscK metric in the class $c_1(L)$ if and only if $(X,L)$ is $K$-polystable \cite{Don1}. 
However, it is a hard
problem, even for toric surfaces, to check $K$-polystability. It contrasts with our explicit construction.
We can compute from the fan describing the toric surface the points to blow-up in order to reach one of
the $X_j's$.
\end{remarkintro}

Corollary~\ref{cori} lies on deep results of Mabuchi \cite{mab} and Stoppa \cite{st} and computation of the Lie algebra characters
associated to Mabuchi's obstruction. This computation is simplified by the symmetries of the surfaces.

\begin{remarkintro}
\rm
 The idea of using symmetries on toric polytopes has also been used by Donaldson \cite{don}, in order
to construct a $9$ points blow-up of $\CC\PP^2$ that is not $K$-polystable for the associated polarization while the Futaki invariant vanishes. In particular, this polarized toric surface admits no cscK metric in the corresponding K\"ahler class.
\end{remarkintro}

We will also compute a bound on the number of blow-ups necessary to get a cscK metric with our construction. We will see
by an example that this bound is far from being optimal, raising the natural question:

\begin{questintro}
 What is the minimal number of blow-ups necessary to obtain a cscK metric on an iterated blow-up of the Hirzebruch surface $\FF_n$ ?
\end{questintro}

We point out that, by Wang and Zhou \cite{wz}, each toric surface admits an extremal metric. This metric is cscK
if and only if the Futaki invariant vanishes \cite{futa}, and an explicit formula for this invariant
described by the polytope is known by the work of Donaldson \cite{don}. 
However, this is far from answering the question since the K\"ahler classes of the extremal metrics
and the K\"ahler classes where the Futaki invariant vanishes might be far from each other.

\subsection{Plan of the paper}
In Section \ref{blup}, we recall results on blow-ups of toric surfaces. Then the proof of Theorem~\ref{theor}
is given in Section \ref{proof}. The last section deals with stability issues related to our construction.

\subsection{Acknowledgments}
I would like to thank Professors Vestislav Apostolov, Paul Gauduchon, Yann Rollin and Michael Singer for their constant supports. I would also like to thank the referee for useful comments.
\section{Blow-ups of toric CSCK surfaces}
\label{blup}

\subsection{Toric surfaces}
A toric surface $X$ is a complex algebraic surface endowed with 
an effective action of a complex torus $T^{\CC}$ of dimension $2$, such that there exists
an open and dense orbit in $X$. Thus $X$ is a compactification of $T^{\CC}$ and is obtained
by gluing divisors to this torus. The structure of $X$ is then entirely described by the way
of glueing these divisors to the torus, and this procedure of gluing is encoded by a fan $\Sigma=\Sigma(X)$ in a $\ZZ$-lattice $N$ of rank $2$ \cite{oda}.
A fan $\Sigma$ is a set of strongly convex rational polyhedral cones $\sigma$ in $N_{\RR}=N\otimes_{\ZZ}\RR$
satisfying the gluing conditions:
\begin{itemize}
 \item If $\sigma\in\Sigma$ and $\tau$ is a face of $\sigma$ then $\tau\in\Sigma$
 \item If $\sigma_1, \sigma_2\in\Sigma$ then $\sigma_1\cap\sigma_2$ is a face of  $\sigma_1$ and $\sigma_2$
\end{itemize}

To each cone $\sigma$ in $\Sigma$, one can associate an affine toric variety $U_{\sigma}=\Spec(\CC[\sigma^\vee\cap M])$,
where $M$ is the dual lattice of $N$ and 
$$\sigma^\vee=\lbrace u\in M_\RR ; \langle u, x \rangle \geq 0\; \text{ for all } x\in\sigma \rbrace.$$
Then the conditions of the fan ensure that these affine charts are glued together and form a toric variety $X$.

This gives a simple description of these surfaces as each of their algebraic properties admits a 
combinatorial counterpart in terms of the fan.

We will denote by $\Sigma^{(i)}(X)$ the set of cones of dimension $i$ in $\Sigma(X)$.
We will identify the rays in $\Sigma^{(1)}(X)$ and primitive generators in $N$ of these rays.

\begin{example}
Let $\FF_n$ be the $n$-th Hirzebruch surface, that is the total space of the fibration 
$$\PP(\mathcal{O}\oplus \mathcal{O}(n)) \rightarrow \CC\PP^1.$$
Note that $\FF_0=\CC\PP^1\times\CC\PP^1$. 
For a suitable coordinate $N\simeq \ZZ^2$, we have $\Sigma^{(1)}(\FF_n)=\lbrace (0,-1),(1,0),(0,1),(-1,-n) \rbrace$.
As an example, $\FF_2$
admits the fan description of Figure 1.

\begin{figure}[htbp]
\label{fig0}
\psset{unit=0.95cm}
\begin{pspicture}(-3.5,-4)(3.5,0)
$$
\xymatrix @M=0mm{
\bullet & \bullet & \bullet & \bullet & \bullet \\
\bullet & \bullet & \bullet & \bullet & \bullet \\
 \bullet & \bullet & \bullet\ar[u]^{e_2}\ar[r]^{e_1}\ar[ldd]\ar[d]& \bullet & \bullet \\
\bullet & \bullet & \bullet & \bullet & \bullet \\
\bullet & \bullet & \bullet & \bullet & \bullet 
}
$$
\end{pspicture}
\caption{$\FF_2$}
\end{figure}

\noindent Here we only represent the elements of $\Sigma^{(1)}(\FF_2)$ by arrows. The other cones of $\Sigma$ are
$\lbrace 0 \rbrace$ and the two dimensional cones delimited by the rays of $\Sigma^{(1)}(\FF_n)$.
We will use this description of a fan in the remaining of this paper.

\end{example}

\subsection{Blow-ups}

The blow-up of a toric surface $X(\Sigma)$ along a torus invariant orbit closure is described in term of the fan $\Sigma$. To each cone
$\sigma\in \Sigma$, one can associate a torus invariant closed subvariety $V_{\sigma}$ as the closure of $\Spec(\CC [\sigma^{\perp}\cap M ])$ in $X$, where $\sigma^{\perp}=\lbrace u\in M_\RR ; \langle u , x \rangle =0 \text{ for all } x \in \sigma \rbrace$. 
Assume that $\sigma$ is of dimension 2, then $V_{\sigma}$ is a torus fixed point. 
Let $\tau$ be the ray in $N_{\RR}$ generated by the sum of the primitive generators of 
one-dimensional faces of $\sigma$. Then $\tau$ induces a subdivision of $\sigma$ into rational strictly
convex polyhedral cones, $\sigma=\sigma'\cup\sigma'' $. Then the blow-up
of $X$ at $V_{\sigma}$ is the toric surface described by the fan $(\Sigma\setminus{\lbrace \sigma \rbrace })\cup{\lbrace \sigma',\sigma'',\gamma\rbrace }$, where $\gamma =\sigma'\cap\sigma''$.

\begin{example} Consider the cone $\sigma_1\in \Sigma(\FF_2)$ generated by $(0,1)$ and $(-1,-n)$ in the lattice $N=\ZZ^2$.
The blow-up of $\FF_2$ at the point $p=V_{\sigma_1}$ is pictured on Figure 2.

\begin{figure}[htbp]
\psset{unit=0.95cm}
\begin{pspicture}(-3.5,-4)(3.5,0)
$$
\xymatrix @M=0mm{
\bullet & \bullet & \bullet & \bullet & \bullet \\
\bullet & \bullet & \bullet & \bullet & \bullet \\
 \bullet & \bullet & \bullet\ar[u]\ar[r]\ar[ld]\ar[ldd]\ar[d]& \bullet & \bullet \\
\bullet & \bullet & \bullet & \bullet & \bullet \\
\bullet & \bullet & \bullet & \bullet & \bullet 
}
$$
\end{pspicture}
\caption{$\rmBl_p(\FF_2)$}
\end{figure}
\end{example}

We will say that a blow-up is \textit{toric} if it is the blow-up at torus fixed points, so that the action
of the torus lifts to the blown-up manifold.
We introduce the following notation: If $X$ is the blow-up of $Y$ at a point $p\in Y$,
we denote $X=\rmBl_p(Y)$. Then by Proposition~\ref{classif},
as $\FF_1\simeq \rmBl_p(\CC\PP^2)$ for $p$ a torus fixed point,
any compact toric surface $X$ different from $\CC\PP^2$ can be written as $\rmBl_{p_m}(\rmBl_{p_{m-1}}\ldots(\rmBl_{p_1}(\FF_n))\ldots)$
for some $n\geq 0$ and $p_{j+1}\in \rmBl_{p_j}(\rmBl_{p_{j-1}}\ldots(\rmBl_{p_1}(\FF_n))\ldots)$. 
We will use the notation $\rmBl_{(p_m,\ldots,p_1)}(\FF_n)=\rmBl_{p_m}(\rmBl_{p_{m-1}}\ldots(\rmBl_{p_1}(\FF_n))\ldots)$
and call it an \textit{iterated blow-up}.

The combinatorial description of singular points and completness for toric surfaces lead to the following result \cite[page 43]{fu}:

\begin{proposition}
\label{classif}
 Each smooth compact toric surface is obtained from an iterated toric blow-up of $\CC\PP^2$
or one of the Hirzebruch surfaces $\FF_n$.
\end{proposition}

\subsection{The Arezzo and Pacard theorem}
In the proof of Theorem~\ref{theor}, we will build a sequence of toric surfaces, each admitting a cscK metric.
This sequence is built by successive blow-ups.
We recall the following theorem of Arezzo and Pacard \cite{ap1} on blow-ups of cscK manifolds:

\begin{theorem}
Let $(X,\om)$ be a compact cscK manifold and let $(p_j)_{j=1\ldots m}$ be $m$ distinct points of $X$.
Assume that $X$ admits no non-zero holomorphic vector field. 
Consider the blow-up of $X$ at $(p_j)$:
$$
\pi: \rmBl_{( p_j ) }(X) \rightarrow X.
$$
Then for each strictly positive real numbers $(a_j)$,
there exists $\ep_0>0$ such that for each $\ep\in(0,\ep_0)$, there is a cscK metric on 
$\rmBl_{(p_j)}(X)$ in the K\"ahler class 
$$
\pi^*[\om]-\ep^2\sum_j a_j \mathrm{PD}(E_j)
$$
where $\mathrm{PD}(E_j)$ denotes the Poincar\'e dual of the exceptional divisor $E_j=\pi^{-1}(p_j)$.
\end{theorem}

In the second paper \cite{ap2}, Arezzo and Pacard proved a similar result when $X$ admits holomorphic vector fields. However,
in this case the blow-up points need to satisfy some stability and genericity conditions \cite[Theorem 1.3.]{ap2}.
Another way to use their theorem in this case is to perform the analysis modulo a finite group of automorphisms
that preserves no holomorphic vector field. We will use a special case of \cite[Theorem 1.4]{ap2}.

\begin{theorem}
\label{blupcscK}
 Let $(X,\om)$ be a compact cscK manifold.
Let $G\subset \Aut(X)$ be a finite group of isometries of $\om$, and $G\cdot p^1=\lbrace p^1_1,\ldots,p^1_{m_1} \rbrace, \ldots, G\cdot p^r=\lbrace p^r_1,\ldots,p^r_{m_r} \rbrace$ distinct orbits of points $( p^1, \ldots, p^r ) \in X^r$.
Assume moreover that no non-zero holomorphic vector field is $G$-invariant.
Consider the blow-up of $X$ at $(p_j^i)$:
$$
\pi: \rmBl_{( p_j^i ) }(X) \rightarrow X.
$$
Then for any strictly positive real numbers $(a_i)_{i=1\ldots r}$,
there exists $\ep_0>0$ such that, for each $\ep\in(0,\ep_0)$, there is a cscK metric $\om_\ep$ on 
$\rmBl_{(p^i_j)}(X)$ in the K\"ahler class 
$$
\pi^*[\om]-\ep^2 \sum_{i,j}  a_i \mathrm{PD}(E^i_j),
$$
where $\mathrm{PD}(E^i_j)$ denotes the Poincar\'e dual of the exceptional divisor $E^i_j=\pi^{-1}(p^i_j)$.
Moreover, $G$ lifts to a finite subgroup $\tilde{G} \subset \Aut(\rmBl_{( p^i_j ) }(X))$ of isometries of $\om_\ep$.
\end{theorem}

\begin{remark}
\rm
The result of Theorem \ref{blupcscK} is slightly different from the result \cite[Theorem 1.4]{ap2}. We can choose the weights $(a_i)$ of the cscK metric on the blown-up manifold to be the same as the initial weights $(a_i)$. This is because under the hypothesis we've chosen, there is no non-zero holomorphic vector field, and thus the analysis of \cite{ap2} reduces to the analysis in \cite{ap1}. In this situation, the weights $(a_i)$ need not to be perturbed for the result to hold.
\end{remark}

\section{Proof of the main theorem}
\label{proof}

Let $X$ be a compact toric surface. We want to show that a torus equivariant blow-up of $X$ admits a cscK metric.
We will first build a sequence of toric surfaces $X_j$ endowed with cscK metrics and then show that 
we can reach one of these $X_j$ by an iterated blow-up of $X$.

\subsection{The surfaces $X_j$}
The $X_j$ are built inductively. Let $X_0=\CC\PP^1\times\CC\PP^1$.
Then $X_{j+1}$ is the blow-up of $X_j$ at its $2^{j+2}$ torus fixed points.
Using the fan description of the toric blow-up process,
the fans describing $\CC\PP^1\times\CC\PP^1$, $X_1$ and $X_2$ are
represented on Figures 3, 4 and 5.

\begin{figure}[htbp]
\label{fig1}
\psset{unit=0.95cm}
\begin{pspicture}(-3.5,-4)(3.5,0)
$$
\xymatrix @M=0mm{
\bullet & \bullet & \bullet & \bullet & \bullet \\
\bullet & \bullet & \bullet & \bullet & \bullet \\
 \bullet & \bullet & \bullet\ar[u]^{e_2}\ar[r]^{e_1}\ar[l]\ar[d]& \bullet & \bullet \\
\bullet & \bullet & \bullet & \bullet & \bullet \\
\bullet & \bullet & \bullet & \bullet & \bullet 
}
$$
\end{pspicture}

\caption{$\CC\PP^1\times\CC\PP^1$}
\end{figure}

\begin{figure}[htbp]
\label{fig2}
\psset{unit=0.95cm}
\begin{pspicture}(-3.5,-4)(3.5,0)
$$
\xymatrix @M=0mm{
\bullet & \bullet & \bullet & \bullet & \bullet \\
\bullet & \bullet & \bullet & \bullet & \bullet \\
 \bullet & \bullet & \bullet\ar[u]\ar[r]\ar[l]\ar[d]\ar[rd]\ar[ru]\ar[ld]\ar[lu]& \bullet & \bullet \\
\bullet & \bullet & \bullet & \bullet & \bullet \\
\bullet & \bullet & \bullet & \bullet & \bullet 
}
$$
\end{pspicture}
\caption{$X_1$}
\end{figure}

\begin{figure}[htbp]
\label{fig3}
\psset{unit=0.95cm}
\begin{pspicture}(-3.5,-4)(3.5,0)
$$
\xymatrix @M=0mm{
\bullet & \bullet & \bullet & \bullet & \bullet \\
\bullet & \bullet & \bullet & \bullet & \bullet \\
 \bullet & \bullet & \bullet\ar[u]\ar[r]\ar[l]\ar[d]\ar[rd]\ar[ru]\ar[ld]\ar[lu]\ar[uur]\ar[urr]
\ar[llu]\ar[uul]\ar[rrd]\ar[rdd]\ar[ldd]\ar[lld]
& \bullet & \bullet \\
\bullet & \bullet & \bullet & \bullet & \bullet \\
\bullet & \bullet & \bullet & \bullet & \bullet 
}
$$
\end{pspicture}
\caption{$X_2$}
\end{figure}

\newpage
\begin{proposition}
\label{prop}
Each of the $X_j$ admits a cscK metric.
\end{proposition}

\begin{proof}
We proceed by induction. First consider a cscK product metric $\om_0$ on $X_0=\CC\PP^1\times\CC\PP^1$ with the same volume on each factor, so that the associated polytope $\Delta_0$ of the K\"ahler manifold is a square (see \cite{gui} for the description of a toric K\"ahler manifold in term of a polytope). This polytope is represented in Figure 6.
\begin{figure}[htbp]
\psset{unit=0.85cm}
\begin{pspicture}(0,-6)(12,0)
\psframe[linecolor=white](0.5,-4.5)(3.5,-1.5)
\psline[linewidth=0.5pt, linestyle=dotted]{<->}(6,0)(6,-6)
\psline{-}(4,-5)(4,-1)
\psline{-}(8,-5)(8,-1)
\psline[linewidth=0.5pt, linestyle=dotted]{<->}(3,-3)(9,-3)
\psline{-}(4,-5)(8,-5)
\psline{-}(4,-1)(8,-1)
\end{pspicture}
\caption{Polytope $\Delta_0$ associated to $(X_0, \om_0)$.}
\end{figure}
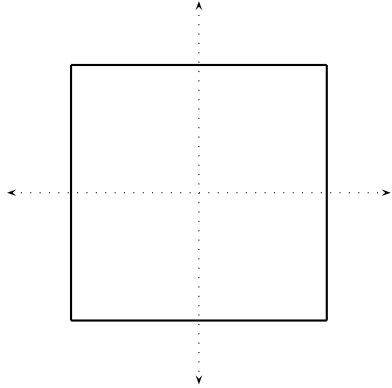
Consider the subgroup $G$ of $\Aut(X_0)=\PP GL_2(\CC)\times \PP GL_2(\CC)$ generated by the following elements, described in terms of
their action in homogeneous coordinates:
$$
\begin{array}{ccc}
([u_1,u_2],[v_1,v_2])&\longmapsto&([u_2,u_1],[v_1,v_2]),\\
([u_1,u_2],[v_1,v_2])&\longmapsto&([u_1,u_2],[v_2,v_1]),\\
([u_1,u_2],[v_1,v_2])&\longmapsto&([-u_1,u_2],[v_1,v_2]),\\
([u_1,u_2],[v_1,v_2])&\longmapsto&([u_1,u_2],[-v_1,v_2]).
\end{array}
$$
By construction, $G$ is a finite group of isometries of $\om_0$ as it is isomorphic to a finite group of symmetries of the associated polytope.
Then the $G$-orbit of $([1,0],[1,0])$ consists of the four fixed points $\lbrace p^0_1,\ldots, p^0_4 \rbrace$ under the torus action on $X_0$ induced by the $\CC^*$-action on each $\CC\PP^1$.
No non-zero holomorphic vector field of $X_0$ is invariant under the $G$ action, thus by Theorem~\ref{blupcscK}, $X_1=\rmBl_{(p^0_1,\ldots,p_4^0)}(X_0)$
admits a cscK metric $\om_1$. Note that the weights of the exceptional divisors in $X_1$ are the same and have to be small by construction. A representation of the polytope $\Delta_1$ associated to $\om_1$ is pictured in Figure 7.
\begin{figure}[htbp]
\psset{unit=0.85cm}
\begin{pspicture}(0,-6)(12,0)
\psframe[linecolor=white](0.5,-4.5)(3.5,-1.5)
\psline[linewidth=0.5pt, linestyle=dotted]{<->}(6,0)(6,-6)
\psline{-}(4,-4.5)(4,-1.5)
\psline{-}(8,-4.5)(8,-1.5)
\psline[linewidth=0.5pt, linestyle=dotted]{<->}(3,-3)(9,-3)
\psline{-}(4.5,-5)(7.5,-5)
\psline{-}(4.5,-1)(7.5,-1)

\psline{-}(4,-4.5)(4.5,-5)
\psline{-}(7.5,-5)(8,-4.5)
\psline{-}(8,-1.5)(7.5,-1)
\psline{-}(4.5,-1)(4,-1.5)

\end{pspicture}
\caption{Polytope $\Delta_1$ associated to $(X_1, \om_1)$.}
\end{figure}
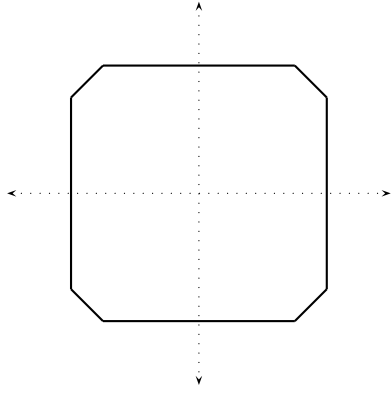
From Theorem \ref{blupcscK}, the metric $\om_1$ admits the lift of $G$ as a subgroup of isometries. Let denote $G_1$ the lift of the subgroup of  $G$ generated by
$$
\begin{array}{ccc}
([u_1,u_2],[v_1,v_2])&\longmapsto&([u_2,u_1],[v_1,v_2]),\\
([u_1,u_2],[v_1,v_2])&\longmapsto&([u_1,u_2],[v_2,v_1]).
\end{array}
$$
This group corresponds to the symmetries of $\Delta_1$ with respect to the two coordinate axes.
The identity component of  the group $\Aut(X_1)$ is the complex torus that defines the toric structure of $X_1$. Then no non-zero holomorphic vector field of $X_1$ is $G_1$-invariant. We apply Theorem \ref{blupcscK} to $(X_1,\om_1)$ blowing up the $G_1$-orbits of the fixed points under the torus action and choosing the weights $(a_i)$ to be all the same. The surface obtained is exactly $X_2$, endowed with a cscK metric $\om_2$. Once again $G_1$ lifts to a group $G_2$ of isometries of $\om_2$ and no non-zero holomorphic vector field on $X_2$ is $G_2$-invariant.
By induction, we build from a cscK metric $\om_j$ on $X_j$, $j\geq 1$, a cscK metric $\om_{j+1}$ on $X_{j+1}$ such that the finite subgroup $G_j$ of isometries of $\om_j$ lifts to a finite subgroup of isometries $G_{j+1}$ of $\om_{j+1}$.  Note that for each $j\geq 1$, 
the Lie algebra of $\Aut(X_j)$ is that of the complex two torus that defines the toric structure, and no non-zero holomorphic vector field is $G_j$-invariant.
By iteration of Arezzo and Pacard theorem, i.e., Theorem \ref{blupcscK}, working modulo $G_j$ at each step and blowing up the $G_j$-orbits of each torus fixed points,
giving the same weight to each new exceptional divisor, we obtain the metrics $\om_j$.
\end{proof}

\subsection{From $X$ to $X_j$}

To end the proof of the theorem, we will use the following two lemmas:

\begin{lemma}
 \label{lemma1}
For any sequence of $m$ toric blow-ups of $\FF_n$,
$\rmBl_{(p_m,\ldots,p_1)}(\FF_n)$,
one of the following holds:
\begin{itemize}
 \item $n=0$,
 \item there is a sequence of $m$ toric blow-ups of $\FF_{n-1}$ such that 
$\rmBl_{(p_m,\ldots,p_1)}(\FF_n)=\rmBl_{(q_m,\ldots,q_1)}(\FF_{n-1})$, or
 \item there is a point $p_{m+1}\in \rmBl_{(p_m,\ldots,p_1)}(\FF_n)$ such that there is a sequence of $m+1$ toric blow-ups 
of $\FF_{n-1}$ such that 
$\rmBl_{(p_{m+1},\ldots,p_1)}(\FF_n)=\rmBl_{(q_{m+1},\ldots,q_1)}(\FF_{n-1}).$
\end{itemize}
\end{lemma}

\begin{lemma}
\label{lemma2}
For any iterated toric blow-up $Y=\rmBl_{(q_m,\ldots,q_1)}(X_0)$ of $X_0$,
there is an iterated toric blow-up of $Y$ such that
$$
\rmBl_{(q'_r,\ldots,q'_1)}(Y)=X_j
$$
for some $j$.
\end{lemma}

The classification of toric surfaces is the key for the proof of the theorem:

\begin{proof}[Proof of Theorem~\ref{theor}]
By Proposition~\ref{classif}, there exists a sequence of toric blow-ups such that
$X=\rmBl_{(p_m,\ldots,p_1)}(\FF_n)$. By iterating Lemma~\ref{lemma1}, there are toric blow-ups $\rmBl_{(q_m,\ldots,q_1)}(X)$ of $X$ such that $\rmBl_{(q_m,\ldots,q_1)}(X)$
is an iterated blow-up of $\FF_0=X_0$. Then by Lemma~\ref{lemma2} we can blow-up $\rmBl_{(q_m,\ldots,q_1)}(X)=\rmBl_{(q'_r,\ldots,q'_1)}(X_0)$ to reach one of the $X_j$.
By Proposition~\ref{prop}, this blow-up of $X$ has a cscK metric.
\end{proof}

We end this section by the proof of the two lemmas:

\begin{proof}[Proof of Lemma~\ref{lemma1}]
 Suppose $n\neq 0$. Consider the fan describing $\FF_n$ on Figure 8.

\begin{figure}[htbp]
\label{fig6}
\psset{unit=0.95cm}
\begin{pspicture}(-4,-4.5)(4,0.5)
$$
\xymatrix @M=0mm{
\bullet & \bullet & \bullet & \bullet & \bullet \\
\bullet & \bullet & \bullet & \bullet & \bullet \\
 \bullet & \bullet & \bullet\ar[u]\ar[r]\ar[ldd]\ar[d]& \bullet & \bullet \\
\bullet & \bullet & \bullet & \bullet & \bullet \\
\bullet & \bullet & \bullet & \bullet & \bullet 
}
$$
\rput[bl]{0}(-3.5,-1.7){$\sigma_1$}

\end{pspicture}
\caption{$\FF_n$}
\end{figure}

Then a blow-up of $\FF_n$ is described by adding a ray generated by the sum of two adjacent primitive vectors
in $\Sigma^{(1)}(\FF_n)$ to this fan. Each further blow-up is described by the same process so we can separate
the blow-ups of $\FF_n$ into two types, those arising by adding a ray in the cone $\sigma_1=\RR^+ (0,1)+\RR^+ (-1,-n)$,
called type one blow-ups, and the other.
These two types of blow-ups commute, we can perform every blow-up of type 1 before the other blow-ups to obtain
$\rmBl_{(p_m,\ldots,p_1)}(\FF_n)$.  Then there are two possibilities. If there is no blow-up of type 1,
we blow-up $V_{\sigma_1}$ and we are in the third case of Lemma~\ref{lemma1}. Indeed, the blow-up of $\FF_n$ at $V_{\sigma_1}$
is a one-point blow-up of $\FF_{n-1}$, (see Figure 9).

\begin{figure}[htbp]
\psset{unit=0.95cm}
\begin{pspicture}(-4,-4.5)(4,0.5)
$$
\xymatrix @M=0mm{
\bullet & \bullet & \bullet & \bullet & \bullet \\
\bullet & \bullet & \bullet & \bullet & \bullet \\
 \bullet & \bullet & \bullet\ar[u]\ar[r]\ar[ldd]\ar[ld]\ar[d]& \bullet & \bullet \\
\bullet & \bullet & \bullet & \bullet & \bullet \\
\bullet & \bullet & \bullet & \bullet & \bullet 
}
$$

\end{pspicture}
\caption{$\rmBl_p(\FF_n)=\rmBl_q(\FF_{n-1})$}
\end{figure}

In the other case, the first blow-up of type one is necessarily 
at $V_{\sigma_1}$ and we are in the second case of Lemma~\ref{lemma2}.
\end{proof}

We need to introduce some definitions. 
We say that an iterated toric blow-up $\rmBl_{(p_m,\ldots,p_1)}(X)$ is a \textit{purely iterated blow-up}
if for each $i\geq 2$, $p_i$ lies on the exceptional divisor $E_{i-1}$ coming from the blow-up
at $p_{i-1}$. In that case, we call $m$ the \textit{length}
of this iterated blow-up.
Let $X$ be a toric surface and $n$ be the minimal integer such that $X$ is obtained from toric blow-ups
of $\FF_n$. Then we define the integer $l(X)$ to be the maximal length of purely iterated blow-ups arising in the 
description of $X$ from $\FF_n$.

\begin{proof}[Proof of Lemma~\ref{lemma2}]
Let $\Sigma^{(1)}(Y)$ be the set of rays of the fan of $Y$, identified with primitive generators of these rays. Then
the fan description of a toric blow-up implies that there is a sequence of fans $\Sigma_j$
such that $\Sigma_0=\Sigma(X_0)$, $\Sigma_m=\Sigma(Y)$ and $\Sigma_{j+1}^{(1)}$ is obtained from $\Sigma_{j}^{(1)}$ by adding a ray generated by $v_i^j+v^j_{i+1}$ , with $v^j_i$ and $v^j_{i+1}$ two adjacent elements of $\Sigma_{j}^{(1)}$.
Consider the set of all $(m+1)$-tuples $(v_0,\ldots,v_m)$ of $\Sigma_{m}^{(1)}$ such that $v_i\in\Sigma^{(1)}_i$,
and for each $i\geq 1$, $v_i=v_{i-1}+w_{i-1}$ with $w_{i-1}\in \Sigma^{(1)}_{i-1}$ adjacent to $v_{i-1}$ or $w_{i-1}=0$.
The case when $w_{i-1}\neq 0$ corresponds to the blow-ups.
Denote $l(v_0,\ldots,v_m)$ the number of distinct elements in $(v_0,\ldots,v_m)$ and
let $l_0(Y)$ be the maximum of the $l(v_0,\ldots,v_m)$. 
Then $l_0(Y)=l(Y)+1$. Then we can blow-up
$Y$ in order to make its fan symmetrical until we obtain the fan describing $X_{l(Y)}$. 
It is enough to add the rays that are missing so that we can get every $(m+1)$-tuple $(v_0,\ldots,v_m)$ with
$l_0(v_0,\ldots,v_m)=l_0(Y)$ which are coming from the set of rays of $X_{l(Y)}$.
An example is pictured on Figure 10.

\begin{example}
 \begin{figure}[htbp]
\psset{unit=0.95cm}
\begin{pspicture}(-4,-4.5)(4,0.5)
$$
\xymatrix @M=0mm{
\bullet & \bullet & \bullet & \bullet & \bullet \\
\bullet & \bullet & \bullet & \bullet & \bullet \\
 \bullet & \bullet & \bullet\ar[u]\ar[ul]\ar[r]\ar[l]\ar[d]\ar[ru]\ar[ruu]\ar[ld]& \bullet & \bullet \\
\bullet & \bullet & \bullet & \bullet & \bullet \\
\bullet & \bullet & \bullet & \bullet & \bullet 
}
$$

\end{pspicture}
\caption{$Y$}
\end{figure}

On this example $l(Y)=2$ because of the sequence $\lbrace (0,1),(1,1),(1,2) \rbrace$.
We blow-up $4$ times to add symmetries, as represented on Figure 11. We blow-up further to reach $X_2$.

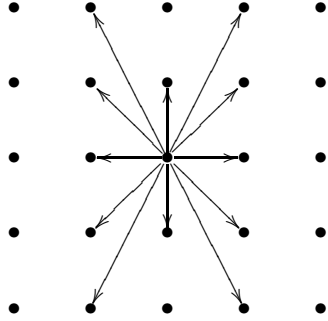
\begin{figure}[htbp]
\psset{unit=0.95cm}
\begin{pspicture}(-4,-4.5)(4,0.5)
$$
\xymatrix @M=0mm{
\bullet & \bullet & \bullet & \bullet & \bullet \\
\bullet & \bullet & \bullet & \bullet & \bullet \\
 \bullet & \bullet & \bullet\ar[u]\ar[ddl]\ar[ddr]\ar[dr]\ar[luu]\ar[lu]
\ar[r]\ar[l]\ar[d]\ar[ru]\ar[ruu]\ar[ld]& \bullet & \bullet \\
\bullet & \bullet & \bullet & \bullet & \bullet \\
\bullet & \bullet & \bullet & \bullet & \bullet 
}
$$
\end{pspicture}
\caption{adding symmetries}
\end{figure}
\end{example}
\end{proof}

\subsection{Bound on the number of blow-ups}

We prove the following upper bound on the minimal number of blow-ups that are needed to get a cscK metric with our method:

\begin{proposition}
 Let $X$ be a toric surface. Let $n$ be the minimal integer such that $X$ is obtained from toric blow-ups
of $\FF_n$. Then there is a sequence of at most 
$$
2^{n+l_0(X)+1}-\sharp (\Sigma^{(1)}(X))
$$
blow-ups of $X$ such that the blown-up surface admits a cscK metric.
\end{proposition}

\begin{remark}
\rm
 Note that for $X=\FF_n$ we have $l_0(X)=1$ and $2^{n+2}-4$ is the number of blow-ups necessary to go from $\FF_n$ to $X_n$
and in that case this estimate is sharp for our method.
\end{remark}

\begin{proof}
 We need to estimate the number of blow-ups necessary to go from $X$ to the nearest $X_j$.
As $X$ is a blow-up of $\FF_n$ for $n$ minimal, there is a $n$ times blow-ups of $X$ such that $\rmBl_{(q_n,\ldots,q_1)}(X)$
is a blow-up of $X_0=\FF_0$. Then from the proof of Lemma~\ref{lemma2},
we know that there exists an iterated blow-up from $\rmBl_{(q_n,\ldots,q_1)}(X)$ to $X_{n+l_0(X)-1}$. 
So we can go from $X$ to $X_{n+l_0(X)-1}$ by an iterated blow-up. Then the number of blow-ups necessary is bounded
by the number of rays in $\Sigma^{(1)}(X_{n+l_0(X)-1})$ minus the number of rays in $\Sigma^{(1)}(X)$, which gives the result.
\end{proof}

This result on the minimal number of points to be blown-up in order to get a cscK metric is not sharp.
Indeed, from \cite[Section 4.1]{rt}, we know that the special blow-up of $\FF_2$ pictured on Figure 12 admits a cscK metric.

\begin{figure}[htbp]
\label{fig9}
\psset{unit=0.95cm}
\begin{pspicture}(-4.5,-6.5)(4.5,0)
$$
\xymatrix @M=0mm{
\bullet & \bullet & \bullet & \bullet & \bullet \\
\bullet & \bullet & \bullet & \bullet & \bullet \\
\bullet & \bullet & \bullet & \bullet & \bullet \\
 \bullet & \bullet & \bullet\ar[r]\ar[ru]\ar[ruu]\ar[rd]\ar[u]\ar[l]\ar[ld]\ar[ldd]\ar[lu]\ar[d]& \bullet & \bullet \\
\bullet & \bullet & \bullet & \bullet & \bullet \\
\bullet & \bullet & \bullet & \bullet & \bullet \\
\bullet & \bullet & \bullet & \bullet & \bullet 
}
$$
\end{pspicture}
\caption{A special blow-up of $\FF_2$}
\end{figure}

Here we do not need to blow-up further in order to reach $X_2$ (Figure 5). It suggests the following problem:

\begin{quest}
 What is the minimal number of blow-ups of $\FF_n$ necessary to ensure the existence of a cscK metric on the blown-up surface?
\end{quest}

\section{Blown-up surfaces and stabilities}
\label{stabi}

In GIT, several notions of stability for polarized varieties are studied. It is in general a hard problem to check 
that a polarized variety is stable or not. However, the existence of a cscK metric enables us to conclude for 
$K$-stability and asymptotic Chow stability in some cases.
In this section, we deduce from Theorem~\ref{theor} that each toric surface can be blown up to satisfy $K$-stability
and asymptotic Chow stability for some polarization.

\subsection{$X_j$ and K-stability}
\label{sec:L}
We refer the reader to \cite{st} for the definition of $K$-stability.
We can assume that the cscK metrics $\om_j$ built on the $X_j$ in Proposition \ref{prop} lie in rational classes. Up to scaling, we can
assume that these metrics represent a polarization $L_j$ of $X_j$, that is $[\om_j]=c_1(L_j)$. By a result of Stoppa \cite{st}, the $(X_j,L_j)$ are $K$-polystable. Together with Theorem~\ref{theor} we have the following corollary.

\begin{corollary}
 Let $X$ be a smooth compact toric surface. Then there is a sequence of toric blow-ups
$$\rmBl_{(p_m,\ldots,p_1)}(X) \rightarrow X$$ and a polarization
$L_X \rightarrow \rmBl_{(p_m,\ldots,p_1)}(X)$ such that $(\rmBl_{(p_m,\ldots,p_1)}(X),L_X)$ is $K$-polystable.
\end{corollary}

\subsection{Asymptotic Chow-polystability of $X_j$}

We refer to \cite{fut} for the definition of asymptotic Chow-stability. By a theorem of Mabuchi \cite{mab},
if $(X,L)$ is a polarized K\"ahler manifold with cscK metric in the class $c_1(L)$ and if Mabuchi's obstruction vanishes,
then $(X,L)$ is asymptotically Chow polystable. From the work of Ono \cite{on} and Futaki \cite{fut},
 Mabuchi's obstruction admits a simple description for toric varieties.
Assume that $(X,L)$ is a polarized toric manifold of complex dimension $n$, with $L$ a torus equivariant ample line bundle. 
Suppose that the lattice $N$ is identified with $\ZZ^n$. Then this polarization defines a polytope $\Delta\subset \RR^n$ that
encodes the symplectic structure of $X$ \cite{oda}. Then set for each $k\in\NN$
$$
\textbf{s}_{\Delta}(k)=\sum_{\textbf{a}\in k\Delta\cap \ZZ^n} \bf{e}(\textbf{a})
$$
and
$$
E_{\Delta}(k)=\sharp (k \Delta \cap \ZZ^n).
$$
Then 
$$
k \mapsto \Vol(\Delta)\textbf{s}_{\Delta}(k)-kE_{\Delta}(k) \int_{\Delta} \textbf{x} dv=\sum_i k^i \mathcal{F}_{\Delta,i}
$$
is a polynomial in $k$ with coefficients $\mathcal{F}_{\Delta,i}$, with $dv$ the euclidian volume of $\RR^n$.
Then Futaki has shown that the vanishing of Mabuchi's obstruction is equivalent to the vanishing of the
$\mathcal{F}_{\Delta,i}$.

\begin{lemma}
 Let $\Delta_j$ be the polytope associated to the toric polarized surface $(X_j,L_j)$ described in Section \ref{sec:L}.
Then $\mathcal{F}_{\Delta_j,1}=\mathcal{F}_{\Delta_j,2}=0$.
\end{lemma}

\begin{proof}
This is a straightforward computation, using the symmetries of $\Delta_j$ with respect to the two coordinate axes.
Note that $\mathcal{F}_{\Delta,1}$ is the Futaki invariant and thus necessarily vanishes by the existence of the cscK metric.
\end{proof}

From this lemma and Mabuchi's theorem we deduce the following corollary.

\begin{corollary}
 Let $X$ be a smooth compact toric surface. Then there is a sequence of toric blow-ups of $X$
$$\rmBl_{(p_m,\ldots,p_1)}(X) \rightarrow X$$ and a polarization
$L_X \rightarrow \rmBl_{(p_m,\ldots,p_1)}(X)$ such that $(\rmBl_{(p_m,\ldots,p_1)}(X),L_X)$ is asymptotically Chow-polystable.
\end{corollary}

\end{document}